\documentclass[12pt,xcolor=svgnames]{amsart}

\usepackage{amsmath, amsthm, amssymb, amsfonts, verbatim}

\usepackage{xspace,mathtools,bm}

\usepackage[svgnames]{xcolor}

\usepackage[pagebackref=true, colorlinks=true, citecolor=blue]{hyperref}

\allowdisplaybreaks

\hypersetup{
    colorlinks=true,
    citecolor=red,
    linkcolor=blue,
    filecolor=magenta,      
    urlcolor=cyan,
    pdfpagemode=FullScreen,
    }
\newtheorem{theorem}{Theorem}[section]

\newtheorem{lemma}[theorem]{Lemma}
\newtheorem{definition}[theorem]{Definition}
\newtheorem{proposition}[theorem]{Proposition}
\newtheorem{remark}[theorem]{Remark}

\newcommand{\R}{\mathbb{R}}
\newcommand{\eps}{\varepsilon}
\newcommand{\abs}[1]{\mid\!#1\!\mid}
\newcommand{\norm}[2]{\left\lVert#1\right\rVert_{#2}}
\newcommand{\eqdef}{:=}

\newcommand{\Obsolete}[1]{
    }

\DeclareMathOperator\supp{supp}

\DeclareFontFamily{U}{mathx}{}
\DeclareFontShape{U}{mathx}{m}{n}{<-> mathx10}{}
\DeclareSymbolFont{mathx}{U}{mathx}{m}{n}
\DeclareMathAccent{\widehat}{0}{mathx}{"70}
\DeclareMathAccent{\widecheck}{0}{mathx}{"71}

\date{\today}
\usepackage[english]{babel}

\author{Zachary Radke}
\address{Oregon State University, Department of Mathematics}
\email{radkeza@oregonstate.edu}
\title{The Inviscid Limit in Uniformly Local Yudovich Spaces}
\date{\today}

\begin{document}

\maketitle
\begin{abstract}
We consider solutions to the two-dimensional Navier-Stokes and Euler equations with the velocity in $L^{\infty}(\R^2)$ and the vorticity in the uniformly local Yudovich space $Y^{\theta}_{ul}(\R^2)$. We show that if $\theta$ grows sufficiently slowly, then for every $T>0$, solutions of the Navier-Stokes equation converge to solutions of the Euler equation in $L^{\infty}([0,T];L^2_{ul}(\R^2))$ as $\nu \to 0^+$, provided they are generated from the same initial data.
\end{abstract}

\section{Introduction}
\noindent In this paper, we study the inviscid limit. The question of the inviscid limit is whether or not a solution of the Navier-Stokes equation converges in some norm to a solution of the Euler equation with the same initial data, as the viscosity tends to $0$. The two-dimensional Navier-Stokes equations modeling incompressible viscous fluid flow are given by
\begin{equation}
\label{NS}
\tag{$NS_{\nu}$}
\begin{cases}
\partial_tu^{\nu}+(u^{\nu}\cdot\nabla)u^{\nu} = \nu\Delta u^{\nu} -\nabla p^{\nu} \quad &(t,x) \in (0,T)\times \R^2,\\
\text{div }u^{\nu}=0  &(t,x) \in [0,T)\times\R^2,\\
u^{\nu}\vert_{t=0}=u_0 & x \in \R^2,
\end{cases}
\end{equation}
where $u^{\nu}:[0,T)\times\R^2\to \R^2$ represents the velocity field of the fluid,  $p^{\nu}:[0,T) \times \R^2 \to \R$ denotes the scalar pressure, and $\nu>0$ is the kinematic viscosity of the fluid. The two-dimensional Euler equations modeling incompressible inviscid fluid flow are given by
\begin{equation}
\label{E}
\tag{$E$}
\begin{cases}
\partial_tu+(u\cdot\nabla)u = -\nabla p \quad &(t,x) \in (0,T)\times \R^2,\\
\text{div }u=0  &(t,x) \in [0,T)\times\R^2,\\
u\vert_{t=0}=u_0 & x \in \R^2.
\end{cases}
\end{equation}
Using the incompressibility condition of the velocity field, for smooth-decaying solutions we can completely determine the pressure from the velocity field using the relation $$p=(-\Delta)^{-1}\text{div}((u\cdot\nabla)u).$$ We also define the \emph{vorticity} of the fluid as curl of the velocity field, i.e. $\omega=\text{curl}(u)=\partial_1u_2-\partial_2u_1$. The velocity can be recovered from the vorticity using the relation $u=-\nabla^{\perp}(-\Delta)^{-1}\omega=K\ast \omega$, where $\nabla^{\perp}=(-\partial_2,\partial_1)$. This is known as the \emph{Biot-Savart Law} and $K$ is the \emph{Biot-Savart kernel}. For a given $\nu \geq 0$, we have the vorticity equation
\begin{equation}
\begin{cases}\label{2d vorticity}
\tag{$V_{\nu}$}
\partial_t\omega^{\nu}+u^{\nu}\cdot\nabla\omega^{\nu}=\nu\Delta\omega^{\nu} &\text{ in } (0,T) \times \R^2,\\
u^{\nu}=-\nabla^{\perp}(-\Delta)^{-1}\omega^{\nu} &\text{ in } [0,T) \times \R^2,\\
\omega^{\nu} \vert_{t=0} = \omega_0 &\text{ in } \R^2,
\end{cases}
\end{equation} 
We use the same convention that when $\nu=0$, we set $\omega^{\nu}=\omega$. Using the method of characteristics, we see that solutions to the two-dimensional Euler equations in vorticity form are given by $\omega(t,x)=\omega_0(X_t^{-1}(x))$ where $X_t^{-1}$ is the inverse of the flow map $X_t(x)$, which satisfies the ODE
\begin{equation}
\label{Flow Map ODE}
\begin{cases}
\frac{d}{dt}X_t(x)=u(t,X_t(x))\quad &\text{in } (0,T)\times\mathbb{R}^2,\\
X_t(x) \vert_{t=0} = x\quad &\text{in } \mathbb{R}^2.
\end{cases}
\end{equation}
As we will be studying solutions to both the Euler and Navier-Stokes equations with low regularity, we must define what we mean by a solution to both. 
\begin{definition}
\label{weak solution}A vector field $u: [0,T) \times \R^2 \to \R^2$ is a weak solution to \eqref{E} with initial data $u(0)=u_0:\R^2  \to \R^2$ if 
\begin{enumerate}
\item $u \in L^2_{loc}([0,T)\times \R^2)$,\\
\item 
$$
\int_0^T\int_{\R^2} \Bigl\{ u\cdot\partial_t\phi + u_ku_l\partial_{x_l}\phi_k\Bigr\}\,dx\,dt = -\int_{\R^2} u_0\cdot\phi(0)\,dx
$$
for all $\phi \in C_c^{\infty}([0,T)\times\mathbb{R}^2)$ with $\text{div }\phi=0$, where we sum over repeated indices, and
\item For each $t \in [0,T)$, $\text{div }u=0$ in $\mathcal{D}'(\R^2)$.
\end{enumerate}
\end{definition}

\begin{definition}
\label{mild solution}
We say that $u^{\nu}:[0,T) \times \R^d \to \R^d$ is a mild solution of \eqref{NS} if
$$
u^{\nu}(t)=e^{\nu \Delta t}u_0 - \int_0^t \text{div}(e^{\nu\Delta (t-s)}
\mathbb{P}(u^{\nu}(s)\otimes u^{\nu}(s))) \,ds \quad \text{ for all } t\in [0,T),$$
where $e^{\nu\Delta t}$ is the heat semigroup defined via the Fourier multiplier $e^{-4\pi^2\nu t |\xi|^2}$, $\mathbb{P}$ denotes the Leray projection with $ij$ component $\delta_{ij}+R_iR_j$, where $\delta_{ij}$ is the Kronecker delta, and $R_i$ is a Riesz transform.
\end{definition}
\subsection{Prior Work}
The question of the inviscid limit for the Navier-Stokes equations has a deep literature, and we do not attempt to include an exhaustive list of results. Instead, we refer the reader to \cite{Constantin2020}, \cite{kato1984remarks},  Appendix III in \cite{temam2024navier}, and the results therein. It turns out that the study of the inviscid limit is intimately tied to the study of uniqueness of solutions. When the vorticity is bounded and decaying, we refer the reader to \cite{yudovich1963non} for a result on uniqueness and \cite{Chemin96} for the inviscid limit. When the vorticity is merely bounded, we refer to \cite{cozzi2009vanishing,cozzi2014vanishing} for results on the inviscid limit. However, when the vorticity is unbounded, establishing uniqueness of weak solutions to \eqref{E} is an extremely delicate problem. In general, the uniqueness of weak solutions with $\omega_0 \in L^p(\R^2)$ for $p<\infty$ is an open problem, and we refer the reader to \cite{BRESSAN20205142}, \cite{brue2023nonuniqueness},  \cite{vishik2018instabilityI}, and \cite{vishik2018instabilityII} for recent progress. Hence, in order to study uniqueness and the inviscid limit for weak solutions with unbounded vorticity, it appears that one must place some additional requirements on the initial data. There appear to be two main methodologies which we discuss below. In the work of Vishik in \cite{vishik1999}, they consider the borderline spaces of Besov type defined by 
\begin{equation}
\label{Borderline Space}
B_{\Pi} = \left\{ f \in \mathcal{S}'(\R^d) : \sum_{j = -1}^N \norm{\Delta_jf}{L^{\infty}} = \mathcal{O}(\Pi(N))\right\}.
\end{equation}
Vishik shows that as long as $\Pi(\alpha)2^{-\alpha}$ is non-increasing for $\alpha \geq C$, $\lim_{\alpha \to \infty} \Pi(\alpha)2^{-\alpha}=0$, and 
\begin{equation}
\label{pi growth}
\int_1^{\infty}\frac{d\alpha}{\Pi(\alpha)}=\infty,
\end{equation}
then if $\omega_{1},\omega_2\in L^{\infty}([0,T];L^{p_0}\cap B_{\Pi}(\R^d))$ for some $1<p_0<d$, then $u_1=u_2$ on $[0,T]$. In \cite{cozzi2007vanishing}, Cozzi and Kelliher establish the vanishing viscosity limit in these borderline spaces of Besov type provided that $\Pi(\alpha)=\log^{\kappa}(\alpha+3)$ for $0 \leq \kappa <1$, and a similar result was proven for axisymmetric Euler without swirl in \cite{cozzi2014axisymmetric}.
\\
\\
Instead of considering vorticity with a condition on the blow-up of the $B^0_{\infty,1}$ norm, another common approach is to place a restriction on the blow-up of the $L^p$ norms of the vorticity. This idea was first considered by Yudovich in \cite{yudovich1995bounded}. We introduce this space below.
\begin{definition}
Let $\theta(p) \geq 1$ be a non-decreasing function on $[1,\infty)$. Then we define the Yudovich space $$Y^{\theta} (\Omega) = \left\{ f \in \bigcap_{2 \leq p<\infty} L^p(\Omega) : \norm{f}{Y^{\theta}(\Omega)}=\sup_{2 \leq p <\infty} \frac{\norm{f}{L^p}}{\theta(p)}<\infty\right\}.$$
\end{definition}
Note that when $\theta(p)=1$ one has $Y^{\theta}(\Omega)=L^2\cap L^{\infty}(\Omega)$. In order to state Yudovich's result, we need to define a class of admissible $\theta$. We consider two functions 
\begin{equation}
\label{Beta definition}
\begin{split}
&\beta_{\eps,M,\phi}(x)=M^{\eps}x^{1-\eps}\phi(1/\eps),\\
&\beta_{M,\phi}(x)=\inf_{\eps \in (0,\frac{1}{2}]}\Bigl\{\beta_{\eps,M,\phi}(x)\Bigr\},
\end{split}
\end{equation}
where $M>0$ is some constant. It is important to note that $\beta_{M,\phi}(x) \to 0$ if and only if $x \to 0^+$.
\begin{definition}
    \label{admissible theta} We say that the function $\theta:[1,\infty) \to \R^+$ is admissible if 
    $$
    \int_0^1 \frac{ds}{\beta_{M,\phi}(s)}=\infty,
    $$
    where $\phi(p)=p\theta(p)$.
\end{definition}
Consequently, every admissible growth function satisfies the integral condition required in Theorem \ref{Taniuchi Main theorem} and Appendix \ref{appendix}.
\begin{lemma}
Let $\theta:[1,\infty)\to[1,\infty)$ be non-decreasing, let $\phi(p)=p\theta(p)$, and let $\beta{_{M,\phi}}$ be as in \eqref{Beta definition}. If 
$$
\int_0^1 \frac{ds}{\beta_{M,\phi}(s)}=\infty,
$$
then
$$
\int_1^{\infty} \frac{dp}{\theta(p)}=\infty.
$$
\end{lemma}
\begin{proof}
We begin by noting that $\beta_{M,\phi}(s)=M\beta_{1,\phi}(s/M)$, so without loss of generality we can assume that $M=1$. For the sake of contradiction, assume that 
$$
\int_1^{\infty}\frac{dp}{\theta(p)}<\infty.
$$
Since $\theta$ is non-decreasing, we see that $\frac{1}{\theta}$ is non-increasing, and hence
$$
\frac{p}{2\theta(p)} \leq \int_{p/2}^p \frac{ds}{\theta(s)} \to 0 \text{ as }p \to \infty.
$$
In particular, one has that $\frac{\theta(p)}{p} \to \infty$. Since $\theta(p) \geq 1$, there exists $C>0$ such that $\theta(p) \geq Cp$ for all $p \geq 2$. Setting $p=\frac{1}{\eps}$ we obtain
\begin{equation*}
\begin{split}
\beta_{1,\phi}(s) &= \inf_{p \geq 2} \left\{s^{1-1/p}p\theta(p)\right\}\\
& \geq Cs \inf_{p \geq 2} \left\{ p^2\exp\left(\frac{\log(1/s)}{p}\right)\right\}.
\end{split}
\end{equation*}
For $0 <s \leq e^{-4}$, the last expression is minimized at $p=\frac{1}{2}\log\left(\frac{1}{s}\right)$ which implies $\beta_{1,\phi}(s) \geq \frac{Ce^2}{4}s(\log(1/s))^2$. It follows that 
$$
\int_0^{e^{-4}} \frac{ds}{\beta_{1,\phi}(s)} \leq \frac{4}{Ce^2}\int_0^{e^{-4}}\frac{ds}{s(\log(1/s))^2}<\infty.
$$
Moreover, we claim that $\beta_{1,\phi(s)} \geq 2s$ for $0<s \leq 1$. To see this, note that 
\begin{equation*}
\beta_{1,\phi}(s) =s\inf_{p \geq 2} s^{-1/p}p\theta(p) \geq s\inf_{p \geq 2} p\theta(p) \geq 2s,
\end{equation*}
where in the first inequality we used that $s^{-1/p} \geq 1$ and in the final inequality we used that $\theta(p) \geq 1$. It follows that the integral over $[e^{-4},1]$ is finite as well, contradicting admissibility.
\end{proof}
An example family of admissible $\theta$ is
\[
\theta_k(p) =  \underbrace{\log(e+p)\cdot \log(e+\log(e+p))\cdot \log(e+\log(e+\log(e+p)))\cdots,}_{\text{$k$ times iterated}}
\] 
where $k \in \mathbb{Z}_{\geq 0}$. Finally, we remark that the condition in Definition \ref{admissible theta} is equivalent to the condition
$$
\int_1^{\infty} \frac{dx}{x\psi(x)}=\infty,
$$
where $\psi(x)=\inf_{\eps \in (0,\frac{1}{2}]}\left\{ \left(x^{\eps}/\eps\right)\theta\left(1/\eps\right)\right\}$ and a change of variables gives 
\begin{equation}
\label{integral equation for theta}
\frac{1}{e}\int_{2}^{\infty} \frac{dp}{p\theta(p)} \leq \int_{e^2}^{\infty} \frac{dx}{x\psi(x)}.
\end{equation}
Hence it is sufficient for $\int_{1}^{\infty}\frac{dp}{p\theta(p)}=\infty$ but not necessary. Under such assumptions on $\theta$, Yudovich establishes the following,
\begin{theorem}[Yudovich, \cite{yudovich1995bounded}]
\label{Yudovich Uniqueness}
Let $\Omega \subseteq \R^2$ be a bounded open set with $C^2$ boundary and assume that $\theta$ is an admissible function. Then for any initial data $\omega_0 \in Y^{\theta}(\Omega)$, there exists a unique weak solution of \eqref{E} such that 
$$
\omega \in L^{\infty}([0,\infty);Y^{\theta}(\Omega)), \quad u \in L^{\infty}([0,\infty);L^{\infty}(\Omega)).
$$
\end{theorem}
The main idea in his proof is as follows. If $\overline{u}$ is the difference of two solutions to the Euler equations generated from the same initial data, one has 
\begin{equation}
\label{Yudovich 1}
\int_{\Omega} |\overline{u}(t,x)|^2\,dx = - \sum_{i,k=1}^2 \int_0^t \int_{\Omega} \left(\frac{\partial u_i}{\partial x_k} + \frac{\partial u_k}{\partial x_i}\right)\overline{u_i}\hspace{.05cm}\overline{u_k}\,dx\,dt.
\end{equation}
When the velocity gradient is bounded, uniqueness follows immediately from Gr\"{o}nwall's Lemma. However, since we are not in this case, Yudovich makes use of the following lemma. 

\begin{lemma}[Lemma 3.1, \cite{yudovich1995bounded}]
\label{beta lemma}
Let $\Omega \subseteq \R^d$ and let $f$ and $g$ be nonnegative, measurable, real functions, such  that $f \in L^1\cap L^{\infty}(\Omega)$  and $ g \in L^{1/\eps}(\Omega)$ for all $\eps \in (0,1/2]$ . Let $M \geq \norm{f}{L^{\infty}(\Omega)}$. Assume that for some positive function $\phi:[1,\infty) \to \R^+$,
$$
\norm{g}{L^p(\Omega)} \leq \phi(p),
$$
for all $p \geq 2$. Then 
$$
\int_{\Omega} f(x)g(x)\,dx \leq \beta_{M,\phi} (\norm{f}{L^1(\Omega)}).
$$
\end{lemma}
\noindent
Then, using that $\norm{\nabla u}{L^p} \leq \frac{Cp^2}{p-1}\theta(p)\norm{\omega}{Y^{\theta}},$ an application of Lemma \ref{beta lemma} to \eqref{Yudovich 1} yields
\begin{equation}
\label{Yudovich 2}
\int_{\Omega} |\overline{u}(t,x)|^2\,dx \leq \int_0^t \beta_{M,\phi}\left(\norm{\overline{u}(t)}{L^2(\Omega)}^2\right)\,dt,
\end{equation}
and uniqueness follows immediately from Osgood's uniqueness theorem and the admissibility condition. Utilizing the above approach that was used to prove uniqueness, in \cite{kelliher2004}, Kelliher was able to establish the inviscid limit in $L^{\infty}([0,T];L^2(\R^2))$ for $u_0 \in L^2(\R^2)$ and $\omega_0 \in Y^{\theta}(\R^2)$, under the same admissibility condition on $\theta$. 
\subsubsection{Uniformly local $L^p$ and Yudovich spaces} As recently observed in \cite{taniuchi2004,taniuchitashirotomoya2010,chen2019two}, the results of Yudovich can be extended to localized spaces as well. We begin by defining such spaces. In what follows we let $\phi \in C_0^{\infty}(\R^2)$ be a standard bump function which is identically $1$ on $B_1(0)$, vanishes on $(B_2(0))^c$, and decreases monotonically when $1 \leq \abs{x}\leq 2$. We also define $\phi_x = \phi(x - \cdot)$ and $\phi_{x,\lambda}=\phi\left(\frac{x-\cdot}{\lambda}\right)$ for any $\lambda>0$.

\begin{definition}
Let $1 \leq p <\infty$. We define the uniformly local $L^p$ space as
$$
L^p_{ul}(\Omega) = \left\{f \in L^p_{loc}(\Omega) : \norm{f}{L^p_{ul}} = \sup_{z \in \R^d} \norm{\phi_zf}{L^p}<\infty\right\}.
$$
For convention we set $L^{\infty}_{ul}(\Omega)=L^{\infty}(\Omega)$.
\end{definition}

\begin{remark} For any fixed $\lambda>0$, we can define an equivalent norm on $L^p_{ul}(\R^d)$ as $$\norm{f}{L^p_{ul}}=\sup_{z \in \R^d}\norm{\phi_{z,\lambda}f}{L^p}.$$ A proof of equivalence can be found in \cite{kato1975cauchy}.
\end{remark} 

\begin{definition} Let $\theta(p) \geq 1$ be a non-decreasing function on $[1,\infty)$.
We define the uniformly local Yudovich space as 
$$
Y^{\theta}_{ul}(\Omega)=\left\{f \in \bigcap_{2 \leq p<\infty} L^p_{ul}(\Omega) : \norm{f}{Y^{\theta}_{ul}} = \sup_{2 \leq p <\infty} \frac{\norm{f}{L^p_{ul}}}{\theta(p)}<\infty\right\}.
$$
\end{definition}
Taniuchi's results can be summarized as follows.
\begin{theorem}
\label{Taniuchi Main theorem}[Taniuchi \cite{taniuchi2004},\cite{taniuchitashirotomoya2010}] Let $\theta:[1,\infty) \to [1,\infty)$ be a non-decreasing function such that 
\begin{equation}
\label{existence growth}
\int_1^{\infty} \frac{dp}{\theta(p)}=\infty.
\end{equation}
Then given $u_0 \in L^{\infty}(\R^2), \text{div }u_0=0, \omega_0 =\text{curl }u_0\in Y^{\theta}_{ul}(\R^2)$, there exists a weak solution $(u,p)$ of \eqref{E}
such that 
$$
\omega \in L^{\infty}_{loc}([0,\infty);Y^{^\theta}_{ul}(\R^2)), u \in L^{\infty}_{loc}([0,\infty);L^{\infty}(\R^2)),
$$
where 
\begin{equation}
\label{Pressure relation}
p=\sum_{k,l=1}^2 R_kR_lu^ku^l \text{ in } BMO(\R^2).
\end{equation}
Moreover, if $\theta$ satisfies 
\begin{equation}
\label{uniqueness growth}
\int_1^{\infty} \frac{dp}{p\theta(p)}=\infty,
\end{equation}
then the solution is unique.
\end{theorem}
\begin{remark} In \cite{taniuchi2004}, the authors state in Remark $1$ that a similar result holds for solutions to the Navier-Stokes equations too. They claim that, if $\theta$ is of the form 
\[
\theta(p) = \underbrace{\log(e+p)\cdot \log(e+\log(e+p))\cdot \log(e+\log(e+\log(e+p)))\cdots,}_{\text{$k$ times iterated}}
\] 
then the solution to \eqref{NS} satisfies
\begin{equation}
\label{NS uly}
\norm{u^{\nu}(t)}{L^{\infty}} \leq C\norm{u_0}{L^{\infty}} \underbrace{\exp C \exp C \cdots \exp}_{\text{$k+2$ times}}\left\{Ct(\norm{u_0}{L^{\infty}} +1)(\norm{\omega_0}{Y^{\theta}_{ul}}+1)^2\right\}.
\end{equation}
In the appendix, we extend this result to non-decreasing $\theta$ satisfying \eqref{existence growth}.
\\
\\
It is also important to note that estimate \eqref{NS uly} is independent of $\nu$. In particular, when we are establishing the inviscid limit, we are able to consider global solutions to the Euler and Navier-Stokes equations with regularity $(u,\omega) \in L^{\infty}(\R^2) \times Y^{\theta}_{ul}(\R^2)$, with uniform bounds that do not depend on the viscosity.
\end{remark}
\subsection{Main Theorem} We state our main theorem below. 
\begin{theorem}
\label{main theorem}
Suppose that $u_0 \in L^{\infty}(\R^2)$ is divergence free and $\omega_0=\text{curl }u_0 \in Y^{\theta}_{ul}(\R^2)$ where $\theta$ is admissible in the sense of Definition \eqref{admissible theta}. 
Let $u^{\nu}$ and $u$ be global mild and weak solutions to \eqref{NS} and \eqref{E} respectively, generated from the same initial data, with $p^{\nu}$ and $p$ satisfying \eqref{Pressure relation}. Then for any $T>0$, the solution of \eqref{NS} converges to the solution of \eqref{E} in $L^{\infty}([0,T];L^2_{ul}(\R^2))$. In particular  
\begin{equation}
\label{main theorem equation}
\lim_{\nu \to 0^+} \sup_{t \in [0,T]} \norm{(u^{\nu}-u)(t)}{L^2_{ul}} = 0.
\end{equation}
\end{theorem}
We remark on a few key difficulties that we must overcome when proving Theorem \ref{main theorem}. The first is establishing a uniformly local estimate on the $L^p$-norms of the velocity gradient which grows like $\frac{p^2}{p-1}\theta(p)$. When the vorticity has decay, one has the relation $u(t)=K\ast\omega(t)$ where $K$ is the Biot-Savart kernel. Note, even if $\omega \in C^{\infty}(\R^2)$, this expression may cease to make sense. Hence, we need to utilize a constitutive law that makes sense for vorticity which lacks spatial decay. This was first done by Serfati in \cite{serfati1995solutions}. The idea is to analyze the evolution of $\partial_tu$, utilize the Biot-Savart law, split the effects of the convolution into near and far field parts, and in the far field use the PDE. Then, integrating in time, one obtains
\begin{equation}
\label{serfati identity 1}
u^j(t)=u_0^j + (aK^j) \ast (\omega(t)-\omega_0) - \int_0^t \nabla \nabla^{\perp}[(1-a)K^j]\ast\cdot (u \otimes u)(s)ds,
\end{equation}
which makes sense for merely bounded vorticity and velocity. We call the relation in \eqref{serfati identity 1} the \emph{Serfati-type identity}. Hence, in order to obtain a desirable estimate on the velocity gradient, we must utilize this identity, which, as a byproduct, results in needing to estimate $L^p$ operator norms of a localized Calder\'{o}n-Zygmund operator. Another challenge is estimating the pressure. In the classical setting one has that 
$$
\int_{\R^2}u(t,x)\cdot\nabla p(t,x)\,dx=0,
$$
which follows from integrating by parts and the incompressibility condition. However, after multiplying by a suitable localization function, this integral does not vanish. To overcome this, we make use of the relation from \cite{taniuchitashirotomoya2010}
\begin{equation}
\label{intro pressure}
p=\sum_{j,k=1}^2 R_jR_ku^ju^k,
\end{equation}
which holds in $BMO(\R^2)$ provided that $u \in L^{\infty}(\R^2)$. As this relation is  a Calder\'{o}n-Zygmund operator of  $u\otimes u$, standard techniques fail, and we will employ a similar near/far-field decomposition strategy in order to estimate the pressure. 
\subsection{Notation and Conventions}
We next define some of the notations that we will use throughout this paper. We let $C>0$ denote an arbitrary constant which may change from line to line. If the constant depends on a set of parameters $\{\lambda_1,\cdots,\lambda_n\}$ and we aim to stress this dependence, we write $C(\lambda_1,\cdots,\lambda_n)$. If there exist constants $C_1$ and $C_2$ such that $A \leq C_1B$ and $B \leq C_2A$, we write $A \sim B$. We  use the convention
$\mathcal{F}(f)(\xi)=\hat{f}(\xi) := \int_{\mathbb{R}^d} f(x) e^{-2\pi i x \cdot \xi}\, dx $ to denote the Fourier transform of a tempered distribution $f:\mathbb{R}^d\to\mathbb{R}$. 
We let $a:\R^2 \to \R$ denote a smooth, radially symmetric bump function such that $a$ is identically $1$ in $B_1(0)$, is monotone decreasing for $1 \leq \abs{x} \leq 2$, and vanishes outside of $B_2(0)$.
\begin{definition}
\label{vector convolution definition}Given two vector fields $v,w$ we define $v \ast \cdot w = \sum_i v^i \ast w^i$. Similarly, for matrix-valued functions $A$ and $B$ on $\R^2$, we define $A\ast\cdot B=\sum_{i,j=1}^dA^{ij}\ast B^{ij}$.
\end{definition}
\section{Preliminaries}
\subsection{ODE Lemmas}

\begin{lemma}\label{Gronwall}
     (Gr\"onwall's Lemma) Let $T>0$ and let $L,\alpha,\beta$ be non-negative, continuous functions on $[0,T]$ with $\alpha$ nondecreasing. If for all $t\in[0,T]$, $$L(t)\leq \alpha(t)+\int_0^t \beta(s)L(s) \, ds,$$ then for all $t\in[0,T]$, $$L(t)\leq \alpha(t)\exp{\left(\int_0^t\beta(s)\,ds\right)}.$$ 
\end{lemma}
We now state the more general Osgood's Lemma.  We refer the reader to \cite{Chemin95}, Chapter 5, for a proof of the lemma. 
\begin{lemma}\label{Osgood}
(Osgood's Lemma) Let $T>0 $, $\Psi$ a positive Borelian function, $\gamma$ a locally integrable positive function, and $\mu$ a continuous increasing function. Assume that for some number $\beta>0$ and for all $t\in[0,T]$, these functions satisfy $$\Psi(t)\leq\beta +\int_0^t\gamma(s)\mu(\Psi(s))\,ds.$$
Then for all $t\in[0,T]$ $$-\phi(\Psi(t))+\phi(\beta)\leq \int_0^t \gamma(s)\,ds,$$ where $\phi(x)=\int_x^1\frac{1}{\mu(r)}\,dr$. 
\end{lemma}

\subsection{Littlewood-Paley Theory} 
Before defining our function spaces, we begin with a brief overview of the Littlewood-Paley operators. We closely follow \cite{Chemin95}.\\
\\
It is well known that if $\mathcal{C}$ is the annulus of center $0$, inner radius $3/5$, and outer radius $5/3$, then there exist two positive radial functions $\hat{\chi} \in C_0^{\infty}(B_{5/6}(0))$ and $\hat{\varphi}\in C_0^{\infty}(\mathcal{C})$ such that, if $j \in \mathbb{Z}$, then we set $\varphi_j(x)=2^{jd}\varphi(2^jx)$ and
$$
\hat{\chi}+\sum_{j\geq 0} \hat{\varphi}_j = \hat{\chi} + \sum_{j \geq 0} \hat{\varphi}(2^{-j}\cdot)
 \equiv 1.$$
For $n\geq 0$ we define $\chi_n$ by $\chi_n(x)=2^{nd}\chi(2^nx)$, so that its Fourier transform satisfies
$$
\hat{\chi}_n(\xi) = \hat{\chi}(2^{-n}\xi) \text{ for } \xi \in \R^d.
$$
For a tempered distribution $f \in \mathcal{S}'(\R^d)$, we define the operator $S_n$ as 
$$
S_n f =\chi_n \ast f.
$$
Formally, $S_n$ can be thought of as a projection in Fourier space into the ball of radius of order $2^n$. For $f \in \mathcal{S}'(\R^d)$ we define the inhomogeneous Littlewood-Paley operators $\Delta_j$ by 
\begin{equation*}
\label{LP operator}
\Delta_j f
=
\begin{cases}
    0, \hspace{1cm} &\text{if }j < -1,\\
    \chi \ast f, &\text{if } j=-1, \\
    \varphi_j \ast f, &\text{if } j\geq 0,
\end{cases}
\end{equation*}
and for all $j\in \mathbb{Z}$ we define the homogeneous Littlewood-Paley operators $\dot{\Delta}_j$ by
$$
\dot{\Delta}_j f = \varphi_j \ast f.
$$
 Up to constants, one can formally think of $\dot{\Delta}_j$ as a projection in Fourier space into the annulus of inner and outer radius of order $2^j$. For any $u \in \mathcal{S}'(\R^d)$, we have that $u = \lim_{n \to \infty} S_nu$ where the limit holds in $\mathcal{S}'(\R^d)$. Or equivalently
$$
u = \sum_{j\geq -1} \Delta_j u \text{ in } \mathcal{S}'(\R^d).
$$
One benefit of using Littlewood-Paley operators can be seen in Bernstein's Lemma. Heuristically, Bernstein's Lemma states that if the support of $\hat{u}$ is included in some ball $B_{\lambda}(0)$, then a derivative of $u$ costs no more than a power of $\lambda$ when taking an $L^p$ norm. This will be made precise below. We let $\mathcal{C}_{a,b}(0)$ denote the annulus with inner radius $a$ and outer radius $b$.

\begin{lemma}(Bernstein's Lemma) Let $(r_1,r_2)$ be a pair of strictly positive real numbers such that $r_1<r_2$ and let $\lambda$ denote a positive real number. Then there exists a constant $C>0$, such that, for every integer $k$ with $k \geq 0$, every pair of real numbers $(a,b)$ satisfying $b \geq a \geq 1$, and every function $u \in L^a(\R^d)$,
\begin{enumerate}

\item 
\label{Bernstein in Ball}
If $\supp{\hat{u}} \subset B_{r_1\lambda}(0)$ then
\begin{equation*}
 \sup_{\abs{\alpha} = k} \norm{\partial^{\alpha}u}{L^b} \leq C^k\lambda^{k+d(\frac{1}{a}-\frac{1}{b})}\norm{u}{L^a}.
\end{equation*}\\
\item 
\label{Bernstein in Annulus}
Furthermore, if $\supp{\hat{u}} \subset \mathcal{C}_{r_1\lambda, r_2\lambda}(0)$ then
\begin{equation*}
 C^{-k}\lambda^k\norm{u}{L^a} \leq  \sup_{\abs{\alpha} = k} \norm{\partial^{\alpha}u}{L^a} \leq C^k\lambda^k\norm{u}{L^a}.
\end{equation*}
\end{enumerate}
\end{lemma}

\subsection{Uniformly Local Estimates}
Throughout the paper, we make extensive use of the following result and use it without remark. This result is well known. See, for example, Lemma $5$ in \cite{lazar2013global}.
\begin{proposition} Suppose that $1 \leq p \leq \infty, \mathcal{K} \in L^1(\R^d),$ and $f \in L^p_{ul}(\R^d)$. Then 
$$
\norm{\mathcal{K} \ast f}{L^p_{ul}} \leq \norm{\mathcal{K}}{L^1}\norm{f}{L^p_{ul}}.
$$
\end{proposition}
When estimating the pressure, we make use of the following lemma.
\begin{lemma}
\label{near field pressure kernel}
Suppose that $G$ is the fundamental solution of the Laplacian, then $\mathcal{F}(\partial_i\partial_j(aG))(\xi) \in L^{\infty}_{\xi}(\R^2)$.
\end{lemma}
\begin{proof}
We use the product rule to write 
\begin{equation}
\label{nfpk product rule}
\partial_i\partial_j(aG)(x) = (\partial_i\partial_ja)G + \partial_ia\partial_jG + \partial_ja\partial_iG + a\partial_i\partial_jG. 
\end{equation}
Recall that $\norm{\hat{f}}{L^{\infty}} \leq \norm{f}{L^1}$, and note that the first three terms on the RHS of \eqref{nfpk product rule} are in $L^1(\R^2)$ since $G\in W^{1,1}_{loc}(\R^2)$. We cannot employ this same technique to handle $a\partial_i\partial_jG$. Instead, we observe that
\begin{equation}
|\mathcal{F}(a\partial_i\partial_jG)(\xi)| \leq C \int_{\R^2} |\hat{a}(\xi-\eta)|\frac{|\eta_i\eta_j|}{|\eta|^2}\,d\eta \leq C\norm{\hat{a}}{L^1}.
\end{equation}
\end{proof}
Next, we aim to establish a growth rate in $p$ of the velocity gradient. Recall that for smooth decaying velocity and vorticity, one has that $\norm{\nabla u}{L^p} \leq \frac{C(d)p^2}{p-1}\norm{\omega}{L^p}$, however, it is not apparent that there is an analogous result in uniformly local spaces. To establish this result, we show that the approximating solution constructed in \cite{taniuchi2004} satisfies the Serfati-type identity, and then pass to the limit. Following, we perform a careful analysis of the operator $Tf=\nabla(aK)\ast f$ in $L^p_{ul}(\R^2)$ where $K$ is the Biot-Savart kernel. It will be crucial to understand the operator norm of $T$ in $L^p_{ul}(\R^2)$ as this will determine the growth rate of $\theta$ that is allowed.
\begin{proposition}
Let $(u,\omega)$ be the weak solution to \eqref{E} as in Theorem \ref{Taniuchi Main theorem}. Then for any $t\geq 0$, the pair $(u,\omega)$ satisfies
\begin{equation}
\label{FLSI}
u(t)=u_0+(aK)\ast(\omega(t)-\omega_0) - \int_0^t \nabla\nabla^{\perp}((1-a)K)\ast\cdot (u\otimes u)(s)\,ds.
\end{equation}
\end{proposition}
\begin{proof}
We begin by recalling that the solution $(u,\omega)$ arises from an approximating sequence $(u^{\eps},\omega^{\eps})$, which is constructed in \cite{taniuchi2004}. They show that $(u^{\eps},\omega^{\eps})$ satisfy the uniform bounds
\begin{equation}
\label{uniform bounds approx}
\begin{split}
\sup_{0 \leq t <T} \norm{u^{\eps}(t)}{L^{\infty}} &\leq C(\norm{u_0}{L^{\infty}},\norm{\omega_0}{Y^{\theta}_{ul}},T),\\
\sup_{0 \leq t <T} \norm{\omega^{\eps}(t)}{L^p_{ul}} &\leq C(\norm{u_0}{L^{\infty}},\norm{\omega_0}{Y^{\theta}_{ul}},T,p),\\ 
\end{split}
\end{equation} and for each $0<\delta<1, 0<t<T$
\begin{equation}
\label{uniform bound cts}
\begin{split}
\norm{u^{\eps}(t)}{C^{1-\delta}}&\leq C(\norm{u_0}{L^{\infty}},\norm{\omega_0}{Y^{\theta}_{ul}},T)\\
\end{split}
\end{equation}
From \eqref{uniform bounds approx} and \eqref{uniform bound cts}, we see that there exists a subsequence $u^{\eps_n}$ such that \begin{equation}
\begin{split}
\label{u convergence}
&u^{\eps_n}\to u \quad \text{weak-$*$ in } L^{\infty}([0,T];L^{\infty}(\R^2))\\
&u^{\eps_n}\to u\quad \text{strongly in } L^{\infty}([0,T];L^{\infty}(B_R(0))) \text{ for each } R>0
\text{ and } T>0.
\end{split}
\end{equation}
Furthermore, $u^{\eps_n}$ satisfies
\begin{equation}
\label{approx sti eq1}
u^{\eps_n}(t)=u_0^{\eps_n} - \int_0^t \mathbb{P}\text{div }(u^{\eps_n}\otimes u^{\eps_n})(s)\,ds.
\end{equation}
In what follows, we aim to show that $u^{\eps_n}$ satisfies a Serfati-type identity, and then pass to the limit. We
let $G(x)=\frac{1}{2\pi}\log|x|$ so that $\Delta G=\delta_0$ and $K=\nabla^{\perp}G$, where $\delta_0$ is the Dirac-delta distribution. It follows that 
\begin{equation*}
\label{matrix kernel calculation}
\begin{split}
(\partial_l^{\perp}K_i)_{i,l=1}^2 &= \begin{pmatrix} \partial_2^2G & -\partial_1\partial_2G \\
-\partial_1\partial_2G & \partial_1^2G \end{pmatrix}\\
&=
\begin{pmatrix} \partial_1^2G+ \partial_2^2G & 0\\
0 & \partial_1^2G +\partial_2^2G\end{pmatrix}-
\begin{pmatrix} \partial_1^2G & \partial_1\partial_2G \\
\partial_1\partial_2G & \partial_2^2G \end{pmatrix}.\\
\end{split}
\end{equation*}
Hence, if $\delta_{il}$ is the Kronecker-delta function, we see that $\partial_l^{\perp}K_i = \delta_{il}\delta_0 - \partial_i\partial_lG $, which is the distributional kernel of $\mathbb{P}_{il}$. Substituting this into \eqref{approx sti eq1} gives
\begin{equation}
\label{approx sti eq2}
\begin{split}
u^{\eps_n}(t)&=u_0^{\eps_n} - \int_0^t \mathbb{P}\text{div }(u^{\eps_n}\otimes u^{\eps_n})(s)\,ds \\
&= u_0^{\eps_n} - \int_0^t \nabla\nabla^{\perp}K\ast\cdot(u^{\eps_n}\otimes u^{\eps_n})(s)\,ds\\
&=u_0^{\eps_n} - \int_0^t \nabla\nabla^{\perp}(aK)\ast\cdot(u^{\eps_n}\otimes u^{\eps_n})(s)\,ds - \int_0^t \nabla\nabla^{\perp}((1-a)K)\ast\cdot(u^{\eps_n}\otimes u^{\eps_n})(s)\,ds.
\end{split}
\end{equation}
Since $u^{\eps_n}$ and $\omega^{\eps_n}$ are smooth, we see that 
$$
\partial_t((aK)\ast\omega^{\eps_n})=(aK)\ast \partial_t\omega^{\eps_n}=-(aK)\ast \text{curl div}(u^{\eps_n}\otimes u^{\eps_n}) = - \nabla\nabla^{\perp}(aK)\ast \cdot(u^{\eps_n}\otimes u^{\eps_n}). 
$$
Then, integrating in time one finds
\begin{equation}
\label{approx sti eq3}
 - \int_0^t \nabla\nabla^{\perp}(aK)\ast\cdot(u^{\eps_n}\otimes u^{\eps_n})(s)\,ds = (aK)\ast (\omega^{\eps_n}(t)-\omega^{\eps_n}_0).
 \end{equation}
 Combining \eqref{approx sti eq2} and \eqref{approx sti eq3} yields
 \begin{equation}
 \label{approx sti eq}
 u^{\eps_n}(t)=u^{\eps_n}_0 + (aK) \ast (\omega^{\eps_n}(t) - \omega_0^{\eps_n}) - \int_0^t \nabla \nabla^{\perp}((1-a)K) \ast \cdot (u^{\eps_n}\otimes u^{\eps_n})(s)\,ds.
 \end{equation}
We aim to pass to the limit in \eqref{approx sti eq}. To do this, we first show a form of local weak convergence for the vorticity. From \eqref{u convergence}, we see that for any $t \in [0,T]$ and $\psi \in C_c^{\infty}(\R^2)$
\begin{equation}
\label{vorticity distributional convergence}
\lim_{n \to \infty} \left |\int_{\R^2} (\omega^{\eps_n}(t,x)-\omega(t,x))\psi(x)\,dx\right| \leq \lim_{n \to \infty} \norm{(u^{\eps_n}(t)-u(t))}{L^{\infty}(\text{supp }(\psi))}\norm{\nabla \psi}{L^1} = 0,
\end{equation}
showing $\omega^{\eps_n}(t) \to \omega(t)$ in $\mathcal{D}'(\R^2)$. Utilizing distributional convergence of the vorticity, we will show that for each $2\leq p<\infty, z \in \R^2$, $\lambda>0$, and $\psi \in L^{p'}(\R^2)$
\begin{equation}
\label{compact weak star convergence}
\lim_{n \to \infty} \int_{\R^2} \phi_{z,\lambda}(x)\omega^{\eps_n}(t,x) \psi(x)dx = \int_{\R^2} \phi_{z,\lambda}(x)\omega(t,x) \psi(x)\,dx.
\end{equation}
 Let $2\leq p <\infty$ , we know that 
\begin{equation}
\label{rhs limit 1}
\sup_{n\in \mathbb{N}} \norm{\omega^{\eps_n}(t)}{L^p(\text{supp }{\phi_{z,\lambda}})} \leq C(\norm{u_0}{L^{\infty}},\norm{\omega_0}{Y^{\theta}_{ul}},T,p).
\end{equation}
Since $\omega^{\eps_n}(t) \to \omega(t)$ in $\mathcal{D}'(\R^2)$, by uniqueness of weak limits and reflexivity, a subsequence (which we relabel) of $\omega^{\eps_n}(t)$ converges weakly to $\omega(t)$ in $L^p(\text{supp }{\phi_{z,\lambda}})$ and by weak lower semi-continuity one has
\begin{equation}
\label{rhs limit 2}
 \norm{\omega(t)}{L^p(\text{supp }{\phi_{z,\lambda}})} \leq C(\norm{u_0}{L^{\infty}},\norm{\omega_0}{Y^{\theta}_{ul}},T,p).
\end{equation}
Let $\psi \in L^{p'}(\R^2)$, by density, we construct a sequence $\{\psi_m\}_{m=1}^{\infty}\subseteq C_c^{\infty}(\R^2)$ such that $\psi_m \to\psi$ in $L^{p'}(\R^2)$. Then  
\begin{equation}
\label{rhs limit 2}
\begin{split}
&\left|\int_{\R^2} \phi_{z,\lambda}(x)(\omega^{\eps_n}(t,x)-\omega(t,x))\psi(x)\,dx\right|  \leq \left|\int_{\R^2} \phi_{z,\lambda}(x)(\omega^{\eps_n}(t,x)-\omega(t,x))\psi_m(x)\,dx\right|\\
&\hspace{3cm}+\left|\int_{\R^2} \phi_{z,\lambda}(x)(\omega^{\eps_n}(t,x)-\omega(t,x))(\psi_m(x)-\psi(x))\,dx\right|=I+II.
\end{split}
\end{equation}
In order to estimate $I$, note that $\phi_{z,\lambda}\psi_m \in C^{\infty}_c(\R^2)$ which implies that 
\begin{equation}
\label{rhs limit 4}
\lim_{n \to \infty} \left|\int_{\R^2} \phi_{z,\lambda}(x)(\omega^{\eps_n}(t,x)-\omega(t,x))\psi_m(x)\,dx\right|=0,
\end{equation}
as $\omega^{\eps_n}(t) \to \omega(t)$ in $\mathcal{D}'(\R^2)$. To estimate $II$, from Hölder's inequality
\begin{equation}
\label{rhs limit 3}
II \leq \sup_{n\in \mathbb{N}}\norm{\phi_{z,\lambda}(\omega^{\eps_n}(t)-\omega(t))}{L^p}\norm{\psi_m-\psi}{L^{p'}} \to 0 \text{ as } m \to \infty,
\end{equation}
where we made use of \eqref{rhs limit 1} and the fact that $\psi_m \to \psi$ in $L^{p'}(\R^2)$, which establishes \eqref{compact weak star convergence}. We now pass to the limit in \eqref{approx sti eq}. Let $x \in \R^2$ be fixed and $2<p<\infty$. Since $\text{supp}((aK)(x-\cdot)) \subseteq B_2(x)$ we note that $aK(x-y)=aK(x-y)\phi_{x,2}(y)$ and hence
\begin{equation}
\label{rhsp1}
\lim_{n \to \infty}\left|(aK) \ast (\omega^{\eps_n}(t)-\omega(t))(x)\right| = \lim_{n \to \infty}\left|\int_{\R^2} (aK)(x-y) \phi_{x,2}(y)(\omega^{\eps_n}(t,y)-\omega(t,y))\,dy \right|=0,
\end{equation}
where we made use of $y \mapsto aK(x-y) \in L^{p'}(\R^2)$. Since $\omega_0^{\eps_n} = \rho_{\eps_n} \ast \omega_0,$ where $\rho_{\eps_n}$ is a standard mollifier, by a similar argument we easily see that 
\begin{equation}
\label{rhsp2}
\lim_{n \to \infty} \left|(aK) \ast (\omega^{\eps_n}_0-\omega_0)(x)\right| \leq \norm{aK}{L^{p'}} \lim_{n \to \infty} \norm{\phi_{x,2}(y)(\omega_0^{\eps_n}(y)-\omega_0(y))}{L^p_y}=0.
\end{equation}
We now estimate the difference coming from the integrand on the right-hand side of \eqref{approx sti eq}. We let $R>0$ to be chosen later. Using \eqref{uniform bounds approx}, one obtains
\begin{equation}
\label{rhsp3}
\begin{split}
&\sup_{0 \leq s \leq T}\left| \nabla\nabla^{\perp}((1-a)K)\ast \cdot(u^{\eps_n}\otimes u^{\eps_n}-u\otimes u)(s,x)\right|\\ &\leq C(\norm{u_0}{L^{\infty}},\norm{\omega_0}{Y^{\theta}_{ul}},T)\int_{|y| \leq R} |\nabla\nabla^{\perp}((1-a)K)(x-y)|\sup_{0 \leq s \leq T}|u^{\eps_n}(s,y)-u(s,y)|\,dy\\
&\hspace{.5cm} +C(\norm{u_0}{L^{\infty}},\norm{\omega_0}{Y^{\theta}_{ul}},T)\int_{|y| > R} |\nabla\nabla^{\perp}((1-a)K)(x-y)|dy\\
&\leq C(\norm{u_0}{L^{\infty}},\norm{\omega_0}{Y^{\theta}_{ul}},T)\left(\norm{\nabla\nabla^{\perp}((1-a)K)}{L^1}\sup_{0 \leq s \leq T} \norm{(u^{\eps_n}-u)(s)}{L^{\infty}(B_R(0))} + \int_{|y| \geq R} \frac{dy}{|x-y|^3}\right)\\
&\leq C(\norm{u_0}{L^{\infty}},\norm{\omega_0}{Y^{\theta}_{ul}},T) \left(\sup_{0 \leq s \leq T}\norm{(u^{\eps_n}-u)(s)}{L^{\infty}(B_R(0))} + \frac{1}{R-|x|}\right).
\end{split}
\end{equation}
Let $\eta>0$, we then choose $R>\max\left\{ 2|x|+4, \frac{4 C(\norm{u_0}{L^{\infty}},\norm{\omega_0}{Y^{\theta}_{ul}},T)}{\eta}\right)$ and use \eqref{u convergence} to choose $N \in \mathbb{N}$ so that $\sup_{0 \leq s \leq T}\norm{(u^{\eps_n}-u)(s)}{L^{\infty}(B_R(0))} \leq \frac{\eta}{2 C(\norm{u_0}{L^{\infty}},\norm{\omega_0}{Y^{\theta}_{ul}},T)}$. It follows that for any $n \geq N$
$$
\sup_{0 \leq s \leq T}\left| \nabla\nabla^{\perp}((1-a)K)\ast\cdot (u^{\eps_n}\otimes u^{\eps_n}-u\otimes u)(s,x)\right| <\eta,
$$
so that 
\begin{equation}
\label{rhsp4}
\lim_{n \to \infty }\left| \int_0^t\nabla\nabla^{\perp}((1-a)K)\ast\cdot (u^{\eps_n}\otimes u^{\eps_n}-u\otimes u)(s,x)\,ds\right|=0.
\end{equation}
Combining \eqref{rhsp1},\eqref{rhsp2}, \eqref{rhsp4}, and passing to the limit for the remaining terms in \eqref{approx sti eq} using \eqref{u convergence} and properties of mollification yields
\begin{equation}
\label{weak sfi}
u(t)=u_0 + (aK) \ast (\omega(t)-\omega_0) - \int_0^t \nabla \nabla^{\perp}((1-a)K) \ast \cdot(u \otimes u)(s)\,ds.
\end{equation}

\end{proof}
When estimating the Serfati-type identity, we need to compute the distributional gradient of $aK$, where $K$ is the Biot-Savart kernel. We closely follow \cite{majda2002vorticity} and utilize the following theorem from \cite{grafakos2008classical}.
\begin{theorem}\label{nice operator}[Theorem 5.3.3, \cite{grafakos2008classical}] Let $\mathcal{K} \in L^1_{loc}(\R^d\setminus\{0\})$ satisfy
\begin{equation}
\label{kernel assumption 1}
\sup_{R > 0} \int_{R \leq |x|\leq 2R} |\mathcal{K}(x)|\,dx=A_1,
\end{equation}
and
\begin{equation}
\label{kernel assupmtion 2}
\sup_{y \neq 0} \int_{|x| \geq 2|y|} |\mathcal{K}(x-y)-\mathcal{K}(x)|\,dx=A_2.
\end{equation}
Let $W \in \mathcal{S}'(\R^d)$ be defined as 
$$
\langle W,\varphi\rangle = \lim_{\eps \to 0^+}\int_{|x| \geq \eps}\mathcal{K}(x)\varphi(x)\,dx.
$$
Let $T$ be the operator associated to convolution with $W$, and assume that $T$ has a bounded extension mapping $L^r(\R^d)$ to itself for some $1<r<\infty$ with norm $B$. Then $T$ has a bounded extension operator from $L^p(\R^d)$ to itself for $1<p<\infty$ with norm
\begin{equation}
\label{operator norm}
\norm{T}{L^p \to L^p} \leq C(d)\max(p,(p-1)^{-1})(A_2+B),
\end{equation}
where $C(d)$ is a constant which does not depend on $p$.
\end{theorem}
\begin{proposition}
\label{Local CZO}
Let $p \in (1,\infty)$, $\mathcal{K} \in C^2(\R^d\setminus\{0\})$ be homogeneous of degree $1-d$, and $f \in L^p_{ul}(\R^d)$. Then 
\begin{equation}
\label{local czo representation}
T_jf=\partial_{x_j} (a\mathcal{K} \ast f) = \text{p.v.}\int_{\R^d}\partial_{x_j} (a\mathcal{K})(x-y)f(y)\,dy + C_0(j,d,\mathcal{K})f,
\end{equation}
where
$$
C_0(j,d,K)=\int_{S^{d-1}}\mathcal{K}(\sigma)\sigma_jdS(\sigma),
$$
and the principal value represents a Calder\'{o}n-Zygmund operator. Furthermore, $T_j$ is a continuous linear operator on $L^p_{ul}(\R^d)$ and 
$$
\norm{T_jf}{L^p_{ul}} \leq \frac{C(j,d,\mathcal{K})p^2}{p-1}\norm{f}{L^p_{ul}}
$$
\end{proposition}
\begin{proof}
We first compute the distributional derivative of $a\mathcal{K}$. Let $\varphi \in C_0^{\infty}(\R^d)$, then
\begin{equation*}
\begin{split}
(\partial_{x_j}(a\mathcal{K}),\varphi) &= -\lim_{\eps \to 0^+} \int_{|x| \geq \eps} a(x)\mathcal{K}(x)\partial_{x_j}\varphi(x)\,dx \\
=& \lim_{\eps \to 0^+} \left\{\int_{|x|\geq \eps} \partial_{x_j}(a(x)\mathcal{K}(x)) \varphi(x)\,dx + \int_{|x|=\eps} a(x)\mathcal{K}(x)\varphi(x)\frac{x_j}{|x|}dS\right\}.
\end{split}
\end{equation*}
In the latter integral, we perform the substitution $x=\eps y$ to see that 
\begin{equation*}
\begin{split}
\lim_{\eps \to 0^+}\int_{|x|=\eps} a(x)\mathcal{K}(x)\varphi(x)\frac{x_j}{|x|}dS &= \lim_{\eps \to 0^+}\int _{|y|=1}a(\eps y) \eps^{1-d}\mathcal{K}(y) \varphi(\eps y) \frac{y_j}{|y|}\eps^{d-1} \,dy\\
&= a(0)\varphi(0) \int_{|y|=1} \mathcal{K}(y)y_j\,dy\\
&=C_0(j,d,\mathcal{K})\varphi(0),
\end{split}
\end{equation*}
where in the second equality we used the Dominated Convergence Theorem, and in the last equality we used that $a(0)=1$, which establishes \eqref{local czo representation}. We make the observation that 
$$
\phi_z(\nabla(a\mathcal{K})\ast f) = \phi_{z}(\nabla(a\mathcal{K})\ast \phi_{z,8}f).
$$
Hence, showing that $T_j$ is a bounded linear operator on $L^p_{ul}(\R^d)$ will follow from showing that $T$ is a bounded linear operator on $L^p(\R^d)$. We first show that $T_j:L^2(\R^d) \to L^2(\R^d)$ is bounded. In particular, we aim to show that $\mathcal{F}(\partial_{x_j}(a\mathcal{K}))(\xi) \in L^{\infty}_{\xi}(\R^d)$ as
$$
\norm{T_jf}{L^2} \leq \norm{\mathcal{F}(\partial_{x_j} (a\mathcal{K}))(\xi)}{L^{\infty}_{\xi}}\norm{f}{L^2}.
$$
To see that $\mathcal{F}(\partial_{x_j}(a\mathcal{K}))(\xi) \in L^{\infty}_{\xi}(\R^d)$, we note that since $\mathcal{K} \in C^2(\R^d\setminus \{0\})$ is homogeneous of degree $1-d$, one has the usual bounds 
\begin{equation}
\label{k decay}
|\mathcal{K}(x)| \leq C(\mathcal{K})|x|^{1-d}, \quad |\nabla \mathcal{K}(x)| \leq C(\mathcal{K})|x|^{-d}, \quad |\nabla^2K(x)|\leq C(K)|\mathcal{K}|^{-d-1}.
\end{equation}
We consider two cases, one when $|\xi| \leq 1$ and the other when $|\xi| >1$. For $|\xi|>1$, we let $\chi$ be a standard bump function, and define $\chi_R(x)=\chi(x/R)$ for $R>0$. We have that 
\begin{equation}
\label{near field ak}
|\widehat{\chi_Ra\mathcal{K}}(\xi)| \leq \norm{\chi_Ra\mathcal{K}}{L^1}\leq C\int_{|x|\leq 2R} |x|^{1-d}\,dx \leq CR.
\end{equation} 
To estimate the far-field, for a fixed $\xi \in \R^d$ we define $e_{\xi}=\frac{\xi}{|\xi|}$ and $\partial_{e_{\xi}}=e_{\xi}\cdot \nabla$. Integration by parts twice yields the estimate
\begin{equation}
\label{far field ak 1}
\begin{split}
\left| \mathcal{F}\left((1-\chi_R)a\mathcal{K}\right)(\xi)\right| &= \left| \int_{\R^d} (1-\chi_R)(x)(a\mathcal{K})(x) e^{-2\pi i x \cdot \xi} \,dx\right|\\
&=
\frac{1}{4\pi^2|\xi|^2}\left| \int_{\R^d} (1-\chi_R)(x)(a\mathcal{K})(x) \partial_{e_{\xi}}^2e^{-2\pi i x \cdot \xi} \,dx\right|\\
&\leq \frac{1}{4\pi^2 |\xi|^2}\norm{\partial_{e_{\xi}}^2((1-\chi_R)(a\mathcal{K}))}{L^1}.
\end{split}
 \end{equation}
Since $a \in \mathcal{S}(\R^d)$, from the product rule we have
 \begin{equation}
 \begin{split}
 \label{far field ak 2}
 \norm{\partial_{e_{\xi}}^2((1-\chi_R)(a\mathcal{K}))}{L^1} & \leq \norm{(1-\chi_R)\partial_{e_{\xi}}^2(a\mathcal{K})}{L^1} + \norm{2(\partial_{e_\xi}\chi_R)(\partial_{e_{\xi}}(a\mathcal{K}))}{L^1} + \norm{ (\partial_{e_{\xi}}^2\chi_R)a\mathcal{K}}{L^1}\\
 &\leq C\int_{R}^{\infty} \rho^{-2}\,d\rho + \frac{C}{R}\int_{R}^{2R} \rho^{-1}\,d\rho + \frac{C}{R^2} \int_{R}^{2R}d\rho\leq \frac{C}{R}.
 \end{split}
 \end{equation}
 Combining \eqref{near field ak}, \eqref{far field ak 1}, \eqref{far field ak 2}, and setting $R=\frac{1}{|\xi|}$ we find that 
 \begin{equation}
\label{xi less than 1 kernel}
|\mathcal{F}((a\mathcal{K}))(\xi)| \leq C(R+\frac{1}{|\xi|^2R})\leq \frac{C}{|\xi|}.
 \end{equation}
 It immediately follows that $\mathcal{F}(\partial_j(a\mathcal{K}))(\xi) \in L^{\infty}_{\xi}(\R^d)$ for $|\xi|>1$. To see the case when $|\xi|\leq 1$, we simply observe that $|\xi||\widehat{a\mathcal{K}}(\xi)| \leq \norm{a\mathcal{K}}{L^1}$ which gives the desired result.
 It remains to show that $\nabla(a\mathcal{K})$ satisfy \eqref{kernel assumption 1} and \eqref{kernel assupmtion 2}. We have that 
\begin{equation}
\begin{split}
\sup_{R\geq 0}\int_{R \leq |x| \leq 2R} |\nabla (a\mathcal{K})|\,dx &\leq \sup_{R \geq 0}\int_{R \leq |x|\leq 2R} |(\nabla a)\mathcal{K}| \,dx + \sup_{R \geq 0 } \int_{R \leq |x| \leq 2R} |a\nabla \mathcal{K}|dx \\
&\leq \int_{1 \leq |x|\leq 2} |(\nabla a)\mathcal{K}| \,dx + \norm{a}{L^{\infty}}\sup_{R \geq 0 } \int_{R \leq |x| \leq 2R} |\nabla \mathcal{K}|dx\\
&\leq C + C\sup_{R \geq 0}\int_{R}^{2R} \frac{1}{r^d}r^{d-1}\,dr = C(1+\ln(2)),
\end{split}
\end{equation}
where in the second inequality we utilized the support of $a$ and $\nabla a$. \eqref{kernel assupmtion 2} follows from the fact that $a \in \mathcal{S}(\R^2)$ and \eqref{k decay}. 
An application of Theorem \ref{nice operator} readily gives
\begin{equation*}
\norm{\phi_zT_jf}{L^p}=\norm{\phi_{z}T_j(\phi_{z,8}f)}{L^p}\leq \frac{C(j,d,\mathcal{K})p^2}{p-1}\norm{\phi_{z,8}f}{L^p},
\end{equation*}
and the result follows from taking a supremum over $z \in \R^d$.
\end{proof}

\begin{theorem}
\label{velocity gradient estimate}
Suppose that a divergence free $u \in L^{\infty}([0,T];L^{\infty}(\R^2))$ and $\omega = \text{curl }u \in L^{\infty}([0,T];L^{p}_{ul}(\R^2))$ satisfy \eqref{FLSI}.
Let $1<p<\infty$, then
$$
\norm{\nabla u(t)}{L^p_{ul}} \leq \frac{C(d)p^2}{p-1}\norm{\omega(t)}{L^p_{ul}}+C\left(\norm{u(t)}{L^p_{ul}}+\norm{u_0}{L^p_{ul}}\right)+Ct\sup_{0 \leq s \leq t}\left\{\norm{u(s)}{L^p_{ul}} \norm{u(s)}{L^{\infty}} \right\}.
$$
\end{theorem}
\begin{proof}
Using a Littlewood-Paley decomposition, we have that 
\begin{equation}
\label{VG 1}
\norm{\nabla u(t)}{L^p_{ul}} \leq \norm{\Delta_{-1}\nabla u(t)}{L^p_{ul}} + \norm{\sum_{j \geq 0} \dot{\Delta}_j\nabla u(t)}{L^p_{ul}}.
\end{equation}
Estimating the low frequencies, one obtains 
\begin{equation}
\label{VG 2}
\phi_z(x)\Delta_{-1}\nabla u(t) = \phi_z(x)\int_{\R^2} \chi(x-y)\nabla u(t,y)\,dy =\phi_z(x) \int_{\R^2} (\nabla \chi)(x-y) u(t,y)\,dy,
\end{equation}
so that 
\begin{equation}
    \label{VG 3}
    \norm{\Delta_{-1} \nabla u(t)}{L^p_{ul}} \leq \norm{\nabla \chi}{L^1}\norm{u(t)}{L^p_{ul}} \leq C\norm{u(t)}{L^p_{ul}}.
\end{equation}
To estimate the high frequencies we use \eqref{FLSI} and that $\dot{\Delta}_ju_0=-\dot{\Delta}_j\nabla^{\perp}(-\Delta)^{-1}\omega_0$ 
\begin{equation}
\begin{split}
\label{VG 4}
\norm{\sum_{j \geq 0} \dot{\Delta}_j\nabla u(t)}{L^p_{ul}} &\leq \norm{\sum_{j \geq 0} \dot{\Delta}_j\nabla u_0 + \sum_{j \geq 0} \dot{\Delta}_j \nabla(aK)\ast (\omega(t)-\omega_0)}{L^p_{ul}}\\
&\hspace{-2cm}+\int_0^t \norm{\sum_{j \geq 0} \dot{\Delta}_j \nabla\nabla\nabla^{\perp}((1-a)K) \ast\cdot (u\otimes u)(s)}{L^p_{ul}}\,ds\\
&\hspace{-2cm}= 
\norm{\sum_{j \geq 0} (\dot{\Delta}_j\nabla K) \ast\omega_0+ \sum_{j \geq 0} \dot{\Delta}_j \nabla(aK)\ast (\omega(t)-\omega_0)}{L^p_{ul}}\\
&\hspace{-2cm}+\int_0^t \norm{\sum_{j \geq 0} \dot{\Delta}_j \nabla\nabla\nabla^{\perp}((1-a)K) \ast\cdot (u\otimes u)(s)}{L^p_{ul}}\,ds\\
&\hspace{-2cm} \leq \norm{\sum_{j \geq 0} \dot{\Delta}_j\nabla((1-a)K)\ast\omega_0}{L^p_{ul}} + \norm{\sum_{j \geq 0} \dot{\Delta}_j \nabla(aK)\ast \omega(t)}{L^p_{ul}}\\
&\hspace{-2cm}+\int_0^t \norm{\sum_{j \geq 0} \dot{\Delta}_j \nabla\nabla\nabla^{\perp}((1-a)K) \ast\cdot (u\otimes u)(s)}{L^p_{ul}}\,ds
\eqdef I+II+III.
\end{split}
\end{equation}
\textbf{Estimate of $I$:} Using that $\omega_0=\nabla^{\perp}\cdot u_0$, associativity of convolution, and slightly abusing notation we have that 
\begin{equation}
    \label{VG 7}
    \begin{split}
I &=  \norm{\sum_{j \geq 0} \dot{\Delta}_j\nabla\nabla^{\perp}((1-a)K)\ast\cdot u_0}{L^p_{ul}}\\
&\leq \norm{\Delta_{-1}\nabla\nabla^{\perp}((1-a)K)\ast\cdot u_0}{L^p_{ul}} + \norm{\nabla\nabla^{\perp}((1-a)K)\ast\cdot u_0}{L^p_{ul}} \\
&\leq C\norm{u_0}{L^p_{ul}}.
    \end{split}
\end{equation}
\textbf{Estimate of $II$:} A direct application of Proposition \ref{Local CZO} yields
\begin{equation}
\label{VG 8}
\begin{split}
II & \leq \norm{\Delta_{-1}\nabla(aK) \ast \omega(t)}{L^p_{ul}} + \norm{\nabla (aK)\ast \omega(t)}{L^p_{ul}} \leq \frac{C(d)p^2}{p-1}\norm{\omega(t)}{L^p_{ul}}.
\end{split}
\end{equation}
\textbf{Estimate of $III$:} We have that 
\begin{equation}
\label{VG 9}
\begin{split}
III & = \int_0^t \norm{\sum_{j \geq 0} \dot{\Delta}_j \nabla\nabla\nabla^{\perp}((1-a)K) \ast\cdot (u\otimes u)(s)}{L^p_{ul}}\,ds\\
&\leq \int_0^t \left\{\norm{\nabla\nabla\nabla^{\perp}((1-a)K) \ast \cdot (u\otimes u)(s)}{L^p_{ul}}+{\norm{\Delta_{-1}\nabla\nabla\nabla^{\perp}((1-a)K) \ast \cdot (u\otimes u)(s)}{L^p_{ul}}} \right\}\,ds\\
&\leq Ct\sup_{0 \leq s \leq t} \left\{\norm{u(s)}{L^{p}_{ul}}\norm{u(s)}{L^{\infty}}\right\}
\end{split}
\end{equation}
Combining \eqref{VG 3}, \eqref{VG 4}, \eqref{VG 7}, \eqref{VG 8}, and \eqref{VG 9} yields
$$
\norm{\nabla u(t)}{L^p_{ul}} \leq \frac{C(d)p^2}{p-1}\norm{\omega(t)}{L^p_{ul}}+C\left(\norm{u(t)}{L^p_{ul}}+\norm{u_0}{L^p_{ul}}\right)+Ct\sup_{0 \leq s \leq t}\left\{\norm{u(s)}{L^p_{ul}} \norm{u(s)}{L^{\infty}} \right\}.
$$
\end{proof}
\begin{remark}
Additionally, if we assume that $\omega = \text{curl }u \in L^{\infty}([0,T];Y^{\theta}_{ul}(\R^2))$, as an immediate corollary of Theorem \ref{velocity gradient estimate}, using that $\frac{p^2}{p-1} \geq 1, \theta(p) \geq 1$, and the regularity of $u$ and $\omega$, we obtain the estimate
\begin{equation}
\label{best velocity gradient estimate}
\sup_{0 \leq t \leq T} \norm{\nabla u(t)}{L^p_{ul}} \leq C(\norm{u_0}{L^{\infty}},\norm{\omega_0}{Y^{\theta}_{ul}}, d, T) p\theta(p)  \quad \text{ for all } 2\leq p<\infty.
\end{equation}
\end{remark}

\section{Proof of Main Theorem}
In what follows, we let $u^{\nu}$ and $u$ be solutions to the Navier-Stokes and Euler equations respectively, generated from the same initial data $u^0$. We see the difference $\overline{u}=u^{\nu}-u$ must satisfy
\begin{equation}
\label{MT1}
\partial_t \overline{u} + (u^{\nu}\cdot \nabla) \overline{u}+(\overline{u}\cdot \nabla) u = \nu \Delta u^{\nu} - \nabla \overline{p},
\end{equation}
where $\overline{p}=p^{\nu}-p$.
We then multiply \eqref{MT1} by $\phi_z^2$ and take an inner product against $\overline{u}$ to obtain
\begin{equation}
    \label{MT2}
    \frac{1}{2}\frac{d}{dt}\norm{\phi_z\overline{u}(t)}{L^2}^2 = I + II + III + IV,
\end{equation}
where
\begin{equation}
\label{MT3}
\begin{split}
I&= - \int_{\R^2} \phi_z^2\overline{u}\cdot( (u^{\nu}\cdot \nabla) \overline{u}) \,dx, \\
II&= -\int_{\R^2} \phi_z^2 ((\overline{u}\cdot\nabla)u)\cdot \overline{u}\,dx,\\
III&= \nu \int_{\R^2} \phi_z^2 \ \Delta u^{\nu}\cdot \overline{u}\,dx, \\
IV&= -\int_{\R^2} \phi_z^2 \nabla \overline{p}\cdot \overline{u}\,dx.
\end{split}
\end{equation}
\\
\\
\textbf{Estimate of $I$:} We integrate by parts to see that
\begin{equation}
\label{estimate of I}
I = \frac{1} {2} \int_{\R^2} \text{div}(\phi_z^2u^{\nu}) |\overline{u}|^2 \,dx \leq \norm{\nabla \phi_z}{L^{\infty}}\int_{\R^2} |u^{\nu}|\phi_z|\overline{u}|^2\,dx.
\end{equation}
\textbf{Estimate of $II$:} We estimate $II$ as 
\begin{equation} 
\label{estimate of II}
II \leq \int |\nabla u| \phi_z^2 |\overline{u}|^2 \,dx.
\end{equation}
\textbf{Estimate of $III$:} Integration by parts yields
\begin{equation}
\label{Estimate of III 1}
\begin{split}
III&= \nu\int_{\R^2}\phi_z^2
\Delta u^{\nu}\cdot \overline{u}\,dx\\
&= \nu \int_{\R^2}\phi_z^2 \Delta \overline{u}\cdot \overline{u}\,dx + \nu \int_{\R^2} \phi_z^2 \Delta u \cdot \overline{u}\,dx\\
&= -\nu \int_{\R^2}\phi_z^2 |\nabla \overline{u}|^2\,dx - 2\nu \int_{\R^2}\phi_z \sum_{i,k=1}^2 (\partial_k\overline{u}_i)(\partial_k\phi_z)\overline{u}_i\,dx \\
&\hspace{.5cm} -\nu \int_{\R^2}\phi_z^2 \nabla u \cdot \nabla \overline{u}\,dx - 2\nu \int_{\R^2}\phi_z \sum_{i,k=1}^2 (\partial_ku_i)(\partial_k\phi_z)\overline{u}_i\,dx.
\end{split}
\end{equation}
From Young's inequality for products we can handle the last three integrals in \eqref{Estimate of III 1} as
\begin{equation}
\label{estimate of III 2}
\begin{split}
& - 2\nu \int_{\R^2}\phi_z \sum_{i,k=1}^2 (\partial_k\overline{u}_i)(\partial_k\phi_z)\overline{u}_i\,dx \leq \frac{\nu}{4}\int_{\R^2}\phi_z^2 |\nabla \overline{u}|^2\,dx + C\nu\int_{\R^2}|\nabla \phi_z|^2 |\overline{u}|^2\,dx,\\
& -\nu \int_{\R^2}\phi_z^2 \nabla u \cdot \nabla \overline{u}\,dx \leq \frac{\nu}{4}\int_{\R^2} \phi_z^2 |\nabla \overline{u}|^2\,dx + C\nu \int_{\R^2}\phi_z^2 |\nabla u|^2\,dx,\\
&- 2\nu \int_{\R^2}\phi_z \sum_{i,k=1}^2 (\partial_ku_i)(\partial_k\phi_z)\overline{u}_i\,dx \leq C\nu \int_{\R^2}\phi_z^2 |\nabla u|^2\,dx + C\nu \int_{\R^2}|\nabla \phi_z|^2|\overline{u}|^2\,dx.
\end{split}
\end{equation} 
Hence, combining \eqref{Estimate of III 1} and \eqref{estimate of III 2} we find 
\begin{equation}
\label{Estimate of III}
III \leq -\frac{\nu}{2}\norm{\phi_z\nabla \overline u(t)}{L^2}^2 + C\nu (\norm{\nabla u(t)}{L^2_{ul}}^2+ \norm{\overline{u}}{L^2_{ul}}^2) \leq C_2(\norm{u_0}{L^\infty},\norm{\omega_0}{Y^{\theta}_{ul}},T)\nu,
\end{equation}
where in the last inequality we used Theorem \ref{Taniuchi Main theorem}, Theorem \ref{velocity gradient estimate}, and the results from Appendix \ref{appendix}.
\\
\\
\textbf{Estimate of $IV$:}
 In order to estimate the pressure we begin by recalling that 
$$
p =   \sum_{i,j=1}^2 -G\ast \partial_i\partial_j (u_iu_j) \quad \text{in } BMO(\R^2),
$$
where $G$ is the fundamental solution of the Laplacian. We then substitute this expression into $IV$ and split the effects of the convolution into near and far fields to obtain 
\begin{equation}
\label{Estiamte of IV p1}
\begin{split}
    IV &= -\int_{\R^2}\phi_z^2 \nabla \left( \sum_{i,j=1}^2 -G \ast \partial_i\partial_j(u_i^{\nu}u_j^{\nu}-u_iu_j)\right)\cdot \overline{u}\,dx\\
    &= -\int_{\R^2} \phi_z^2\nabla \left(\sum_{i,j=1}^2 (-aG) \ast \partial_i\partial_j(u_i^{\nu}u_j^{\nu}-u_iu_j)\right)\cdot \overline{u}\,dx\\&\hspace{1cm}-  \int_{\R^2} \phi_z^2 \sum_{i,j=1}^2\left( \partial_i\partial_j\nabla(-(1-a)G) \ast (u_i^{\nu}u_j^{\nu}-u_iu_j)\right)\cdot \overline{u}\,dx\\ 
    &\eqdef IV_{near} + IV_{far}.
 \end{split}
\end{equation}
We estimate the near and far field terms separately. For the near field estimate, we integrate by parts, use the incompressibility condition, and utilize the support of $\phi_z$ to see  
\begin{equation}
\label{near field pressure estimate 1}
\begin{split}
IV_{near} &= \int_{\R^2} (\nabla (\phi_z^2)\cdot \overline{u}) \sum_{i,j=1}^2 \left((\partial_i\partial_j(-aG) \ast (u_i^{\nu}u_j^{\nu}-u_iu_j)\right) \,dx\\
&\leq C\sum_{i,j=1}^2 \int_{\R^2} \phi_z\phi_{z,2}|\overline{u}|\left| \left(\partial_i\partial_j (-aG) \ast (u_i^{\nu}u_j^{\nu}-u_iu_j)\right)\right|\,dx\\
&\leq C \norm{\overline{u}}{L^2_{ul}}\sum_{i,j=1}^2\norm{\partial_i\partial_j(-aG) \ast (u_i^{\nu}u_j^{\nu}-u_iu_j)}{L^2_{ul}}.
\end{split} 
\end{equation}
In order to estimate the last term in \eqref{near field pressure estimate 1} we first note that for any $z \in \R^2$ one has
$$
\phi_z\left(\partial_i\partial_j(-aG) \ast (u_i^{\nu}u_j^{\nu}-u_iu_j)\right)=\phi_z\left(\partial_i\partial_j(-aG) \ast \phi_{z,8}(u_i^{\nu}u_j^{\nu}-u_iu_j)\right),
$$
which implies
\begin{equation}
\label{near field pressure estimate 2}
\norm{\partial_i\partial_j(-aG) \ast (u_i^{\nu}u_j^{\nu}-u_iu_j)}{L^2_{ul}} \leq \sup_{z \in \R^2}\norm{\left(\partial_i\partial_j(-aG) \ast \phi_{z,8}(u_i^{\nu}u_j^{\nu}-u_iu_j)\right)}{L^2},
\end{equation}
where we used $|\phi_z(x)| \leq 1$. Using a Littlewood-Paley decomposition, Lemma \ref{near field pressure kernel}, and Plancherel's Theorem, one obtains
\begin{equation}
\label{near field pressure estimate 3}
\begin{split}
\norm{\left(\partial_i\partial_j(-aG) \ast \phi_{z,8}(u_i^{\nu}u_j^{\nu}-u_iu_j)\right)}{L^2}^2 &= C\sum_{k \geq -1} \norm{\Delta_k\left(\partial_i\partial_j(-aG) \ast \phi_{z,8}(u_i^{\nu}u_j^{\nu}-u_iu_j)\right)}{L^2}^2\\
&= 
C\sum_{k \geq -1} \norm{\mathcal{F}\left(\partial_i\partial_j(-aG)\right)\mathcal{F}\left(\Delta_k \phi_{z,8}(u_i^{\nu}u_j^{\nu}-u_iu_j)\right)}{L_{\xi}^2}^2\\
&\leq C\norm{\mathcal{F}\left(\partial_i\partial_j(-aG)\right)}{L^{\infty}_{\xi}}^2\sum_{k \geq -1} \norm{\Delta_k \left(\phi_{z,8}(u_i^{\nu}u_j^{\nu}-u_iu_j)\right)}{L^2}^2\\
&\leq
C\norm{u_i^{\nu}u_j^{\nu}-u_iu_j}{L^2_{ul}}^2.
\end{split}
\end{equation}
It follows then that 
\begin{equation}
\label{near field pressure estimate 4}
IV_{near} \leq C \norm{\overline{u}}{L^2_{ul}}^2 \left(\norm{u^{\nu}}{L^{\infty}}+\norm{u}{L^{\infty}}\right).
\end{equation}
We next estimate the far-field term of the pressure. One has that 
\begin{equation}
\label{far field pressure estimate}
\begin{split}
IV_{far} &= -\int_{\R^2}\phi_z^2  \left(\sum_{i,j=1}^2 \partial_i\partial_j\nabla\left((-(1-a)G)\ast(u_i^{\nu}u_j^{\nu}-u_iu_j)\right)\right)\cdot \overline{u}\,dx\\
&\leq C\norm{\overline{u}}{L^2_{ul}} \sum_{i,j=1}^2 \norm{\partial_i\partial_j\nabla ((-(1-a)G))\ast\cdot (u_i^{\nu}u_j^{\nu}-u_iu_j)}{L^2_{ul}}\\
&\leq 
C\norm{\overline{u}}{L^2_{ul}}\sum_{i,j=1}^2 \norm{\partial_i\partial_j\nabla ((-(1-a)G))}{L^1}\norm{u_i^{\nu}u_j^{\nu}-u_iu_j}{L^2_{ul}}\\
&\leq C\norm{\overline{u}}{L^2_{ul}}^2 \left(\norm{u^{\nu}}{L^{\infty}}+\norm{u}{L^{\infty}}\right).
\end{split}
\end{equation}
Combining \eqref{MT2}, \eqref{estimate of I},\eqref{estimate of II},\eqref{Estimate of III},\eqref{near field pressure estimate 4}, \eqref{far field pressure estimate}, we find that 
\begin{equation}
\label{MT4}
\frac{d}{dt}\norm{\overline{u}(t)}{L^2_{ul}}^2 \leq C\sup_{z \in \R^2}\int_{\R^2} |\nabla u(t)|\phi_z^2|\overline{u}(t)|^2\,dx + C_1(t)\norm{\overline{u}(t)}{L^2_{ul}}^2 + C_2(\norm{u_0}{L^\infty},\norm{\omega_0}{Y^{\theta}_{ul}},T)\nu.
\end{equation}
where $C_2$ is defined in \eqref{Estimate of III} and $C_1(t)=\norm{u^{\nu}(t)}{L^{\infty}}+\norm{u(t)}{L^{\infty}}$. An application of Lemma \ref{beta lemma} and Theorem \ref{velocity gradient estimate} to \eqref{MT4} yields
\begin{equation}
\label{MT5}
\frac{d}{dt} \norm{\overline{u}(t)}{L^2_{ul}}^2 \leq C\beta_{M,\psi}(\norm{\overline{u}(t)}{L^2_{ul}}^2) + C_1(t)\norm{\overline{u}(t)}{L^2_{ul}}^2 + C_2(\norm{u_0}{L^\infty},\norm{\omega_0}{Y^{\theta}_{ul}},T)\nu,
\end{equation}
where 
$$
\psi(p)=C(d,T,\norm{u_0}{L^{\infty}},\norm{\omega_0}{Y^{\theta}_{ul}},T)p\theta(p)
,$$
and 
$$
M=\sup_{0 < \nu \leq 1}\sup_{0 \leq s \leq T}\left(\norm{u(s)}{L^{\infty}}+\norm{u^{\nu}(s)}{L^{\infty}}\right)^2,
$$
which is admissible in the sense of Definition \ref{admissible theta}. For ease of notation, we set $L(t)=\norm{\overline{u}(t)}{L^2_{ul}}^2$ and take the substitution $Y(t)=e^{-\int_0^tC_1(s)\,ds}L(t)$. It then follows that 
\begin{equation}
\label{ODE for Y}
Y'(t) \leq e^{-\int_0^t C_1(s)\,ds}\left(C\beta_{M,\psi}(e^{\int_0^tC_1(s)\,ds}Y(t))+C_2(\norm{u_0}{L^\infty},\norm{\omega_0}{Y^{\theta}_{ul}},T)\nu\right).
\end{equation}
From the definition of $\beta_{M,\psi}$ one obtains
\begin{equation}
\label{beta comparison}
\begin{split}
Ce^{-\int_0^t C_1(s)\,ds}\beta_{M,\psi}(e^{\int_0^t C_1(s)\,ds}Y(t)) &= Ce^{-\int_0^t C_1(s)\,ds} \inf_{\eps \in (0,1/2]} \left\{ M^{\eps} (e^{\int_0^t C_1(s)\,ds}Y(t))^{1-\eps} \psi(1/\eps)\right\}\\
&\leq C_3 \inf_{\eps \in (0,1/2]} \left\{ M^{\eps} (Y(t))^{1-\eps} \psi(1/\eps)\right\}\\
&= C_3\beta_{M,\psi} (Y(t)).
\end{split}
\end{equation}
 Substituting \eqref{beta comparison} into \eqref{ODE for Y} and integrating in time yields
\begin{equation}
\label{C eqn}
Y(t) \leq C_2(\norm{u_0}{L^\infty},\norm{\omega_0}{Y^{\theta}_{ul}},T)T\nu  + C_3\int_0^t \beta_{M,\psi}(Y(s))\,ds.
\end{equation}
An application of Osgood's Lemma gives
\begin{equation}
\label{Final Osgood Estimate}
\int_{C_2(\norm{u_0}{L^\infty},\norm{\omega_0}{Y^{\theta}_{ul}},T)T\nu}^1 \frac{ds}{C_3\beta_{M,\psi}(s)} \leq t\ + \int_{Y(t)}^1 \frac{ds}{C_3\beta_{M,\psi}(s)}.
\end{equation}
Taking the limit as $\nu \to 0^+$, using the admissibility condition, we see that the left-hand side of \eqref{Final Osgood Estimate} must go to positive infinity; hence, so must the right. Since $\beta_{M,\psi}(s)$ is strictly positive when $s$ is bounded away from zero, we see that 
$$
\lim_{\nu \to 0^+}\sup_{t \in [0,T]} Y(t)=0,
$$
which in turn yields that 
\begin{equation}
\label{convergence}
\lim_{\nu \to 0^+}\sup_{t \in [0,T]} \norm{\overline{u}(t)}{L^2_{ul}}=0.
\end{equation}
\begin{remark}
Regarding the rate of convergence in \eqref{convergence}, we obtain an analogous result to \cite{kelliher2004} and provide an overview of their result and how it relates to our result. We implicitly define the function $f:\R^+ \to \R^+$ by 
$$
\int_x^{f(x)} \frac{ds}{C_3\beta_{M,\psi}(s)}=T.
$$
From the admissibility condition, we claim that $\lim_{x \to 0^+}f(x)=0$. To see this, assume for contradiction that there exists a sequence $\{x_n\}_{n=1}^{\infty}$ and a $\delta>0$ such that $f(x_n) \geq \delta$ for all $n \in \mathbb{N}$. Then,
$$
T=\int_{x_n}^{f(x_n)} \frac{ds}{C_3\beta_{M,\psi}(s)} \geq \int_{x_n}^{\delta}\frac{ds}{C_3\beta_{M,\psi}(s)} \to \infty  \quad \text{as } n\to \infty.
$$
Hence it must be the case that $f(x) \to 0$ as $x \to 0^+$.  Now, from the Osgood argument \eqref{Final Osgood Estimate} we have that 
$$
\int_{C_2(\norm{u_0}{L^\infty},\norm{\omega_0}{Y^{\theta}_{ul}},T)T\nu}^{Y(t)} \frac{ds}{C_3\beta_{M,\psi}(s)} \leq T,
$$
which in particular implies that $Y(t) \leq f(C_2(\norm{u_0}{L^\infty},\norm{\omega_0}{Y^{\theta}_{ul}},T)T\nu)$ so that 
\begin{equation}
\label{rate of convergence}
\norm{\overline{u}(t)}{L^2_{ul}} \leq \left(e^{\int_0^t C_1(s)ds}f\left(C_2(\norm{u_0}{L^\infty},\norm{\omega_0}{Y^{\theta}_{ul}},T)T\nu\right)\right)^\frac{1}{2}.
\end{equation}
\end{remark}
\begin{appendix}
\section{Uniform Bounds for Navier-Stokes}
\label{appendix}
\noindent
Below, we consider non-decreasing $\theta:[1,\infty) \to [1,\infty)$ satisfying
$$
\int_1^{\infty}\frac{dp}{\theta(p)}=\infty.
$$
We consider a solution pair $(u^{\nu},\omega^\nu)$ to the Navier-Stokes equation with initial data $u_0 \in L^{\infty}(\R^2), \text{div }u_0=0$, and $\omega_0 = \text{curl }u_0 \in Y^{\theta}_{ul}(\R^2)$. We aim to show that there exists constants $C_T$ and $M_T$ (which do not depend on the viscosity) such that 
\begin{equation}
\label{NS uniform bounds}
\sup_{0 < \nu \leq 1}\sup_{0 \leq t \leq T}\norm{u^{\nu}(t)}{L^{\infty}} \leq M_T \text{ and } \sup_{0 < \nu \leq 1}\sup_{0 \leq t \leq T}\norm{\omega^{\nu}(t)}{Y^{\theta}_{ul}} \leq C_T.
\end{equation}
We closely follow the results of \cite{taniuchi2004}. Assume that $u^{\nu}$ is a solution of the Navier-Stokes equations in mild-formulation in the sense of Definition \ref{mild solution}, i.e. assume that
\begin{equation}
\label{ns mild formualtion}
u^{\nu}(t)=e^{\nu\Delta(t-s)}u^{\nu}(s)-\int_s^t e^{\nu\Delta(t-\tau)}\mathbb{P}\nabla \cdot (u^{\nu} \otimes u^{\nu})(\tau)\,d\tau, \quad \text{for all } 0 <s \leq t.
\end{equation}
 We then split $u^{\nu}$ at low and high frequencies. To estimate the low-frequencies, we note that  $$
\norm{S_Ne^{\nu\Delta(t-\tau)}\mathbb{P}\nabla \cdot (f\otimes f)}{L^{\infty}} \leq C2^N\norm{f}{L^{\infty}}^2 \quad \text{for all } f \in L^{\infty}(\R^2),
$$
which yields the estimate 
\begin{equation}
\label{ns low frequency velocity bound}
\norm{S_Nu^{\nu}(t)}{L^{\infty}} \leq C_0\norm{u^{\nu}(s)}{L^{\infty}}+C2^N\int_s^t \norm{u^{\nu}(\tau)}{L^{\infty}}^2\,d\tau.
\end{equation}
To estimate the high frequencies, we utilize Proposition 1.6 from \cite{taniuchi2004} to see that 
\begin{equation}
\label{ns high frequency velocity bound}
\sum_{j \geq N} \norm{\dot{\Delta}_ju^{\nu}(t)}{L^{\infty}} \leq \sum_{j \geq N} 2^{-j}\norm{\dot{\Delta}_j\omega^{\nu}(t)}{L^{\infty}} \leq C2^{-N}\max\left\{2^{2N/p},1\right\}\norm{\omega^{\nu}(t)}{L^p_{ul}}.
\end{equation}
From Lemma 1.4 in \cite{taniuchi2004} one has the uniform viscosity bound
\begin{equation}
\label{uniform vorticity viscosity}
\norm{\omega^{\nu}(t)}{L^p_{ul}} \leq C(1+(t-s)(1+\sup_{s\leq \tau \leq t}\norm{u^{\nu}(\tau)}{L^{\infty}}))^{2/p}\norm{\omega^{\nu}(s)}{L^p_{ul}}.
\end{equation}
Then, combining \eqref{ns low frequency velocity bound},\eqref{ns high frequency velocity bound}, and \eqref{uniform vorticity viscosity} we see that 
\begin{equation}
\begin{split}
\label{ns velocity bound 1}
\norm{u^{\nu}(t)}{L^{\infty}}\leq&   C_0\norm{u^{\nu}(s)}{L^{\infty}}+C_12^N\int_s^t\norm{u^{\nu}(\tau)}{L^{\infty}}^2\,d\tau\\ &+ C2^{-N}\max\left\{2^{2N/p},1\right\}(1+(t-s)(1+\sup_{s\leq \tau \leq t}\norm{u^{\nu}(\tau)}{L^{\infty}}))^{2/p}\norm{\omega^{\nu}(s)}{L^p_{ul}}.
\end{split}
\end{equation}
We set $g(t)=1+\sup_{0 \leq s \leq t}\norm{u^{\nu}(s)}{L^{\infty}}$, which we note is continuous and non-decreasing. Since $2C_0>1$, one has that $g(t)<2C_0g(s)$ for $t$ sufficiently close to $s$. We define
$$
T_1(s)=\sup\left\{ \tau \geq 0: g(s+\tau) \leq 2C_0g(s)\right\},
$$
so that 
\begin{equation}
\label{g pre iterate}
g(t) \leq 2C_0g(s) \quad \text{for all } s \leq t \leq s+T_1(s).
\end{equation}
Let $\eps>0$. Then, for $s \leq t\leq s+\min\left\{1,T_1(s)\right\}$ and $p \geq 2+\eps$ one has 
\begin{equation}
\label{ns velocity bound 2}
g(t) \leq C_0g(s) + 4C_0^2C_12^N(g(s))^2(t-s)+C(\eps)2^{-N}\max\left\{2^{2N/p},1\right\}(g(s))^{2/p}\norm{\omega^{\nu}(s)}{L^p_{ul}}.
\end{equation}
Following an identical argument to \cite{taniuchi2004}, we then find that 
\begin{equation}
\label{ns velocity bound 3}
g(t) \leq 2C_0g(s) \quad \text{for all } s\leq t \leq s+\frac{C_0}{2C(\eps)}\min\left\{\frac{1}{\norm{\omega^{\nu}(s)}{L^p_{ul}}^{p/(p-2)}}, \frac{1}{g(s)^{2/p}\norm{\omega^{\nu}(s)}{L^p_{ul}}},1\right\}.
\end{equation}
We set $\eps=2, p_k=k+4,$ and define the sequence $\{t_k\}_{k=0}^\infty$ as $t_0=0$ and 
\begin{equation}
\label{tk definition}
t_{k+1}=t_k + \frac{1}{2C}\min\left\{\frac{1}{\norm{\omega^{\nu}(t_k)}{L^{p_k}_{ul}}^{p_k/(p_k-2)}}, \frac{1}{(g(t_k))^{2/p_k}\norm{\omega^{\nu}(t_k)}{L^{p_k}_{ul}}},1\right\} \quad \text{ for all } k \geq 0.
\end{equation}
Hence, from \eqref{ns velocity bound 3} we see that 
\begin{equation}
\label{g iterate}
g(t_k) \leq 2C_0g(t_{k-1}) \leq \cdots \leq (2C_0)^kg(0).
\end{equation}
Substituting \eqref{g iterate} into \eqref{uniform vorticity viscosity} and using that $t_k \leq k$ yields
\begin{equation}
\label{ns vorticity uniform 2}
\begin{split}
\norm{\omega^{\nu}(t_k)}{L^{p_k}_{ul}} 
&\leq C(1+t_kg(t_k))^{2/p_k}\norm{\omega^{\nu}_0}{L^{p_k}_{ul}}\\
&\leq C((k+1)g(t_k))^{2/(k+4)}\norm{\omega^{\nu}_0}{L^{p_k}_{ul}}\\
&\leq C(k+1)^{2/(k+4)} (g(0))^{1/2}\theta(p_k)\norm{\omega^{\nu}_0}{Y^{\theta}_{ul}}\\
&\leq C(g(0))^{1/2}\theta(p_k)\norm{\omega^{\nu}_0}{Y^{\theta}_{ul}}.
\end{split}
\end{equation}
We now note that $1<\frac{p_k}{p_k-2}\leq 2$ which gives the estimate
\begin{equation}
\label{ns vorticity uniform 3}
\norm{\omega^{\nu}(t_k)}{L^{p_k}}^{p_k/(p_k-2)}\leq Cg(0)(\norm{\omega_0}{Y^{\theta}_{ul}}+1)^2\theta(p_k)^{p_k/(p_k-2)}.
\end{equation} Since $\theta \geq 1$ we also obtain
\begin{equation}
\label{ns vorticity uniform 4}
\begin{split}
g(t_k)^{2/p_k}\norm{\omega^{\nu}(t_k)}{L^{p_k}_{ul}} &\leq Cg(t_k)^{2/p_k} g(0)^{1/2}\theta(p_k)\norm{\omega_0}{Y^{\theta}_{ul}}\\
&\leq Cg(0)\theta(p_k)\norm{\omega_0}{Y^{\theta}_{ul}}\\
&\leq Cg(0)(1+\norm{\omega_0}{Y^{\theta}_{ul}})^2\theta(p_k)^{p_k/(p_k-2)},
\end{split}
\end{equation} where in the second inequality we used \eqref{g iterate}. From \eqref{tk definition}, \eqref{ns vorticity uniform 3}, and \eqref{ns vorticity uniform 4} we see that 
\begin{equation}
\label{t_k difference bound}
\begin{split}
t_{k+1}-t_k&=\frac{1}{2C}\min\left\{\frac{1}{\norm{\omega^{\nu}(t_k)}{L^{p_k}_{ul}}^{p_k/(p_k-2)}}, \frac{1}{(g(t_k))^{2/p_k}\norm{\omega^{\nu}(t_k)}{L^{p_k}_{ul}}},1\right\}\\
&\geq \frac{C}{g(0)(1+\norm{\omega_0}{Y^{\theta}_{ul}})^2\theta(p_k)^{p_k/(p_k-2)}}.
\end{split}
\end{equation}
We now show that $t_k \to \infty$ from the condition $\int_1^{\infty}\frac{dp}{\theta(p)}=\infty$. We see that 
$$
t_k=\sum_{j=1}^k(t_j-t_{j-1}) \geq \frac{C}{g(0)(1+\norm{\omega_0}{Y^{\theta}_{ul}})^2}\sum_{j=0}^{k-1} \frac{1}{\theta(j+4)^{(j+4)/(j+2)}}.
$$
We will show that $\sum_{n=4}^{\infty}\frac{1}{\theta(n)^{n/(n-2)}}=\infty$. We split $\mathbb{Z}_{\geq 4}$ into the disjoint sets $\left\{ n \geq 4:\theta(n) \leq e^n\right\}$ and $ \left\{ n \geq 4: \theta(n)>e^n\right\}$. Clearly
$$
\sum_{\substack{n \geq 4\\ \theta(n) > e^n}} \frac{1}{\theta(n)} \leq \sum_{n=1}^{\infty}e^{-n}<\infty, 
$$
and hence 
$$
\sum_{\substack{n \geq 4\\ \theta(n)\leq e^n}} \frac{1}{\theta(n)} =\infty.
$$
For $n \geq 4$ satisfying $\theta(n) \leq e^n$ we observe that 
\begin{equation*}
\begin{split}
\frac{1}{\theta(n)^{n/(n-2)}} &= \frac{1}{\theta(n)} \exp\left(-\frac{2\log(\theta(n))}{n-2}\right)\\
&\geq \frac{1}{\theta(n)}\exp\left(-\frac{2n}{n-2}\right) \geq \frac{e^{-4}}{\theta(n)}.
\end{split}
\end{equation*}
It follows that 
$$
\sum_{n =4}^{\infty} \frac{1}{(\theta(n))^{n/(n-2)}}=\infty.
$$
We define the sequence of partial sums 
$$
S_N=\sum_{k=0}^{N-1} \frac{1}{\theta(k+4)^{(k+4)/(k+2)}}.
$$
so that $t_N \geq \frac{CS_N}{g(0)(1+\norm{\omega_0}{Y^{\theta}_{ul}})^2}.$ Then, given $T<\infty$, we choose $N \gg 1$ such that $\frac{CS_N}{g(0)(1+\norm{\omega_0}{Y^{\theta}_{ul}})^2} \geq T$. Since $t_N \geq T$ it follows that
\begin{equation}
\label{ns final uniform velocity}
\sup_{0 \leq t \leq T}\norm{u^\nu(t)}{L^{\infty}} \leq (2C_0)^N (1+\norm{u_0}{L^{\infty}}) \eqdef M_T.
\end{equation}
Then for $p \geq 2$ we similarly obtain
\begin{equation}
\label{ns final uniform vorticity}
\frac{\norm{\omega^{\nu}(t)}{L^p_{ul}}}{\theta(p)} \leq C(1+TM_T)^{2/p}\norm{\omega_0}{Y^{\theta}_{ul}} \leq  C(1+TM_T)\norm{\omega_0}{Y^{\theta}_{ul}} \eqdef C_T,
\end{equation}
and the result follows from taking a supremum over $p \geq 2$.

\end{appendix}
\section*{Acknowledgements}
\noindent The author would like to thank Elaine Cozzi for suggesting this problem to work on and for continued mentoring. The author also thanks Yasushi Taniuchi for fruitful discussions regarding their work.
\bibliographystyle{plain}
\bibliography{YulVVbib}

\end{document}